\documentclass[12pt,a4paper,intlimits,draft]{amsproc}
\usepackage{cite}

\newcommand{\abs}[1]{\lvert#1\rvert}
\newcommand{\Abs}[1]{\biggl\vert{#1}\biggr\vert}
\theoremstyle{plain}
\newtheorem{theorem}{Theorem}
\newtheorem{corollary}{Corollary}
\newtheorem{lemma}{Lemma}
\newtheorem{proposition}{Proposition}
\theoremstyle{remark}
\newtheorem*{remark}{Remark}

\title[Bose--Einstein distribution and Bose condensation]
{On the Bose--Einstein distribution\\ and Bose condensation}
\author{V.P.~Maslov}
\address{Faculty of Physics, Moscow State University\newline\indent
Institute for Problems in Mechanics, Russian Academy of
Sciences}
\email{v.p.maslov@mail.ru}
\author{V.E.~Nazaikinskii}
\address{Institute for Problems in Mechanics, Russian Academy of
Sciences}
\email{nazaikinskii@yandex.ru}

\begin{document}

\begin{abstract}
For a system of identical Bose particles sitting on integer energy
levels, we give sharp estimates for the convergence of the sequence
of occupation numbers to the Bose--Einstein distribution and for
the Bose condensation effect.
\end{abstract}

\maketitle

\tableofcontents

\section{Introduction}\label{s1}

The Bose--Einstein distribution has been studied by physicists
(here we only mention the well-established textbooks by Landau and
Lifshits \cite{LaLi51} and Kvasnikov~\cite{Kva02}) as well as
mathematicians (e.g., see Vershik's papers \cite{Ver96,Ver97},
where further references can be found). It is well known that the
vector of occupation numbers tends, in a certain sense, to the
Bose--Einstein distribution as the total energy of the system and
the number of particles tend to infinity. It is, however, of
interest to establish sharp estimates for this convergence,
including the case of Bose condensation. This has been done in our
papers \cite{MN1-1,MN1-2,MN1-3,MN2-1,MN2-2}. The present exposition
is mainly based on these papers but contains a number of
improvements and has the important advantage of being largely
self-contained (at least as far as Bose condensation itself and
distributions with variable number of particles are concerned).
Moreover, in contrast to other mathematical treatments of the
subject, it does not use any but very elementary mathematical
tools. (For example, we avoid resorting to number-theoretic results
like the Meinardus theorem \cite{And76} or methods of analytic
number theory \cite{Kar04}.)

\medskip

Before proceeding to the results themselves, let us recall what
Bose condensation is by using an elementary combinatorial model,
namely, that of balls distributed over boxes.

Suppose that there is a sequence of boxes $U_j$, $j=0,1,2,\dots$,
and each box $U_j$ is divided into $q_j$ compartments. We take $N$
identical balls and put them into the boxes at random observing the
only condition that
\begin{equation}\label{e1}
    \sum_{j=0}^\infty jN_j\le M,
\end{equation}
where $N_j$ is the number of balls in the box $U_j$ and $M$ is a
positive integer specified in advance. As an outcome, we obtain a
sequence of nonnegative integers $N_j$, $j=0,1,2,\dots$, such that
\begin{equation}\label{e2}
    \sum_{j=0}^\infty N_j=N
\end{equation}
and condition \eqref{e1} is satisfied. It is easily seen that,
given $M$ and $N$, there are finitely many such sequences. Suppose
that all allocations of balls to compartments are equiprobable.
Since the number of ways to distribute $N_j$ indistinguishable
balls over $q_j$ compartments is equal to
\begin{equation*}
    \binom{N_j+q_j-1}{N_j}
    =\frac{(q_j+N_j-1)!}{N_j!(q_j-1)!},
\end{equation*}
it follows that each sequence $\{N_j\}$ can be realized in
\begin{equation}\label{e3}
  W(\{N_j\})=\prod_{j=0}^{\infty}\binom{N_j+q_j-1}{N_j}
\end{equation}
ways, and the probability of this sequence is equal to $W(\{N_j\})$
divided by the sum of expressions similar to \eqref{e3} over all
sequences of nonnegative integers satisfying the constraints
\eqref{e1} and \eqref{e2}. This makes the set of all such sequences
a probability space; the corresponding probability measure will be
denoted by $\mathsf{P}_{M,N}$. The positive integers $q_j$ are
called the \textit{multiplicities}. We will assume that
\begin{equation}\label{e4}
    q_j=Qj^{d-1}+o(j^{d-1}), \quad j\to\infty,
\end{equation}
where $d>1$ is a given parameter (which we refer to as
\textit{dimension}) and $Q\ge1$ is a positive constant.

What happens as $M,N\to\infty$? It turns out that the so-called
\textit{condensation phenomenon} can occur: if $N$ tends to
infinity too rapidly, namely, if it exceeds some threshold
$\overline{N}=\overline{N}(M)$, then a majority of the excessive
$N-\overline{N}$ balls end up landing in the box $U_0$; more
precisely, with probability asymptotically equal to $1$, the number
of balls in $U_0$ is close to $N-\overline{N}$ (and accordingly,
the total number of balls in all the other boxes is close to
$\overline{N}$, now matter how large $N$ itself is). Let us state
the corresponding assertion (which is a special case of
Corollary~\ref{cor4} and the subsequent argument in
Sec.~\ref{3.2.2}).
\begin{theorem}\label{t0}
Define $\overline{N}=\overline{N}(M)$ by the formula
\begin{equation}\label{e5}
    \overline{N}=\sum_{j=1}^\infty\frac{q_j}{e^{bj}-1},
\end{equation}
where $b$ is the unique positive root of the equation
\begin{equation}\label{e6}
    \sum_{j=1}^\infty\frac{jq_j}{e^{bj}-1}=M,
\end{equation}
and set
\begin{equation*}
 \Delta=\begin{cases}
         (\overline{N}\ln\overline{N})^{1/2}\chi(\overline{N})&\text{if $d>2$},\\
          \overline{N}^{1/d}\ln\overline{N}\chi(\overline{N})&\text{if $1<d\le 2$},
        \end{cases}
\end{equation*}
where $\chi(x)$, $x\ge0$, is an arbitrary positive function
arbitrarily slowly tending to infinity as $x\to\infty$. There exist
constants $C_m$, independent of $M$ and $N$, such that if $N>\overline{N}$,
then
\begin{equation*}
    \mathsf{P}_{M,N}(\vert N_0-(N-\overline{N})\vert>\Delta)\le C_m\overline{N}^{-m},
    \qquad m=1,2,\dotsc\,.
\end{equation*}
\end{theorem}

Nothing of this sort happens if $N\le \overline{N}$. In this case,
the limit distribution as $M,N\to\infty$ is the Bose--Einstein
distribution with parameters $\beta,\mu>0$ (see formulas \eqref{v7}
and \eqref{v7a} below), and no condensation on the zero level
occurs.

It is not hard to write out an asymptotic formula
for~$\overline{N}$. To this end, one substitutes \eqref{e4} into
\eqref{e5} and \eqref{e6} and applies the Euler--Maclaurin formula
to the resulting series so as to transform them into integrals. The
result (see formula~\eqref{p0} in Sec.~\ref{3.1.2}) is that
\begin{equation}\label{e11}
    \overline{N}=C(d)M^{\frac d{d+1}}Q^{\frac1{d+1}}
    (1+o(1)),
\end{equation}
where $C(d)$ is a constant\footnote{The explicit expression is
\begin{equation*}
    C(d)=\frac{\Gamma(d)\zeta(d)}
    {(\Gamma(d+1)\zeta(d+1))^{\frac d{d+1}}},
\end{equation*}
where  $\Gamma(x)$ is the gamma function and $\zeta(x)$ is the
Euler zeta function.
}
depending only on the dimension~$d$.

\medskip

Suppose that, for given $M$ and $N>\overline{N}$, we wish to avoid
Bose condensation. How can we do that?

One way would be to increase $\overline{N}$ so as to ensure that
$\overline{N}\ge N$. To this end, let us partly ``Boltzmannize''
the system, i.e., make the balls partly distinguishable. More
precisely, suppose that the model is basically the same, but we are
additionally allowed to paint each of the $N$ balls at random into
one of $K$ distinct colors. Now that we can distinguish between
balls of different colors (but balls of a same color are still
indistinguishable), the Bose condensation threshold $\overline{N}$
should change.

Let us compute how exactly it changes. To make the computation,
instead of painting the balls, we mentally divide each of the $q_j$
compartments in the $j$th box into $K$ sub-compartments and put the
uncolored balls there (with the understanding that putting a ball
into the $k$th sub-compartment is equivalent to painting the ball
into the $k$th color). Now there are $Kq_j$ compartments in the
$j$th box, and we see that, all in all, the introduction of $K$
colors has the only effect that all multiplicities $q_j$ are
multiplied by $K$.

Let us apply Theorem~\ref{t0} (with $q_j$ replaced by the new
multiplicities $\widetilde q_j=Kq_j$). Formula~\eqref{p0} gives an
asymptotic expression for the new threshold, which we denote by
$\widetilde N$. All we have to do is to replace $Q$ by $KQ$ in
formula~\eqref{e11}; then we obtain
\begin{equation*}
    \widetilde N=C(d)M^{\frac d{d+1}}(KQ)^{\frac1{d+1}}(1+o(1))
         =K^{\frac1{d+1}}\overline{N}(1+o(1)).
\end{equation*}

Thus, the introduction of $K$ distinct colors has raised
$\overline{N}$ by the factor $K^{\frac1{d+1}}$.

\section[Main Results]%
{Main Results}\label{s3}

In this section, we state our main results. All proofs are given in
Sec.~\ref{s4}. The simplest physical model to imagine behind our
mathematical constructions is that of a system of identical Bose
particles sitting on integer energy levels.

\subsection{System with a variable number of particles}\label{3.1}

We start our analysis by considering the ``photonic'' case, where
the number of particles in the system is not fixed and only a
constraint on the overall system energy is given.

\subsubsection{Definition of the system}\label{3.1.1}
Let $M\ge0$ be an integer. We denote the set of all sequences
$\{N_j\}\equiv\{N_j\}_{j=1}^\infty$ of nonnegative integers
satisfying the condition
\begin{equation}\label{f2}
    \sum_{j=1}^{\infty} jN_j\le M
\end{equation}
by~$\Omega_M$. Note that all such sequences are finitely
supported\footnote{That is, all but finitely many $N_j$ are zero.
Indeed, $N_j=0$ for $j>M$.} and $\Omega_M$ is finite.

Next, let positive real numbers $q_j>0$, $j=1,2,\dotsc$, be given.
We introduce the probability space $\mathcal{X}_M=(\Omega_M,\mathcal{F}_M,\mathsf{P}_M)$,
where $\mathcal{F}_M=2^{\Omega_M}$ is the powerset of~$\Omega_M$ (as is customary
with discrete probability spaces) and the probability~$\mathsf{P}_M$ is
defined as follows. We assign the weight
\begin{equation}\label{f3}
    w(\{N_j\})=\prod_{j=1}^{\infty}\binom{N_j+q_j-1}{N_j},
\end{equation}
where $\binom zn$ is the generalized binomial
coefficient,\footnote{Recall that $\binom
zn=\prod_{j=0}^{n-1}\frac{z-j}{n-j}$, the empty product (for $n=0$)
being by definition equal to~$1$. Since $\{N_j\}$ is finitely
supported, it follows that only finitely many factors in \eqref{f3}
are different from~$1$.} to each element $\{N_j\}\in\Omega_M$ and the
weight
\begin{equation}\label{f9d}
    w(\mathcal{A})=\sum_{\{N_j\}\in\mathcal{A}}w(\{N_j\})
\end{equation}
to each subset~$\mathcal{A}\subset\Omega_M$ and set\footnote{In what follows,
we also feel free to write the condition determining the set~$\mathcal{A}$
instead of the argument $\mathcal{A}$ itself in expressions like
$\mathsf{P}_M(\mathcal{A})$.}
\begin{equation}\label{f4}
    \mathsf{P}_M(\mathcal{A})=\frac{ w(\mathcal{A})}{w(\Omega_M)},\qquad \mathcal{A}\subset\Omega_M.
\end{equation}

The numbers $q_j$ will be referred to as
\textit{multiplicities}.\footnote{If the $q_j$ are integers, then
$\binom{N_j+q_j-1}{N_j}$ is exactly the number of ways in which $N_j$
indistinguishable particles can be placed on an energy level of
multiplicity $q_j$, and $w(\{N_j\})$ is the number of distinct
system states corresponding to the sequence $\{N_j\}$ of occupation
numbers.} We assume that they have the asymptotics\footnote{The
asymptotics~\eqref{f1a} with $Q=\Gamma(d)^{-1}$ holds, for example,
for the multiplicities $q_j=\binom{j+d-1}{j}$ of energy levels of
the $d$-dimensional quantum-mechanical harmonic oscillator.}
\begin{equation}\label{f1a}
    q_j=Qj^{d-1}(1+o(1))\qquad\text{as $j\to\infty$}
\end{equation}
for some real constants $d>1$ and $Q>0$. In particular, there
exists constants $B_1,B_2>0$ such that\footnote{Most of the results
stated below, except for some explicit formulas like~\eqref{p0},
remain valid if we drop~\eqref{f1a} and only require that the
estimates~\eqref{f1} be true.}
\begin{equation}\label{f1}
    B_1j^{d-1}\le q_j\le B_2j^{d-1},\quad j=1,2,\dotsc\,.
\end{equation}

\subsubsection{Limit distribution}\label{3.1.2}

It turns out that if $M$ is large, then a randomly chosen sequence
$\{N_j\}\in\Omega_M$ is with high probability close, in the sense
described in Theorem~\ref{t1} below, to the nonrandom sequence
$\{\overline{N}_j\}$ defined as follows. For $M>0$, the equation
\begin{equation}\label{f5}
    \sum_{j=1}^{\infty}\frac{jq_j}{e^{bj}-1}=M
\end{equation}
has a unique solution $b>0$ (which tends to zero as
$M\to\infty$),\footnote{Indeed, it follows from~\eqref{f1} that the
series on the left-hand side in \eqref{f5} converges for all $b>0$;
moreover, the sum of this series is easily seen to be a function
monotone decreasing from $\infty$ to $0$ as $b$ goes from $0$ to
$\infty$.} and we set
\begin{equation}\label{f6}
    \overline{N}_j=\frac{q_j}{e^{bj}-1}.
\end{equation}

In the theorem below, we also need the sum
\begin{equation}\label{f7}
    \overline{N}=\sum_{j=1}^{\infty}\overline{N}_j\equiv\sum_{j=1}^{\infty}\frac{q_j}{e^{bj}-1},
\end{equation}
which tends to infinity together with $M$. Take a positive function
$\chi(x)$, $x\ge0$, tending (arbitrarily slowly) to infinity as
$x\to\infty$ and set\footnote{One should bear in mind that $b$,
$\overline{N}_j$, $\overline{N}$, and $\Delta$ are functions of $M$, even though we do
not always write~out the~argument $M$ explicitly.}
\begin{equation}\label{f8}
    \Delta=\begin{cases}
    (\overline{N}\ln\overline{N})^{1/2}\chi(\overline{N})&\text{if $d>2$},\\
    \overline{N}^{1/d}\ln\overline{N}\chi(\overline{N})&\text{if $1<d\le 2$}.
    \end{cases}
\end{equation}

Now we are in a position to state our assertions.
\begin{theorem}\label{t1}
Let the numbers $q_j$ satisfy \eqref{f1a}. Then there exist
constants $C_s>0$, $s=1,2,\dotsc$, such that the estimates
\begin{equation}\label{f9}
    \mathsf{P}_M\biggl(\Abs{\sum_{j=1}^{\infty} f_j(N_j-\overline{N}_j)}>\Delta\biggr)\le C_s\overline{N}^{-s},
    \qquad s=1,2,\dots,
\end{equation}
where $\overline{N}_j$, $\overline{N}$, and $\Delta$ are defined in \eqref{f6},
\eqref{f7}, and \eqref{f8}, hold for an arbitrary $M>0$ and an
arbitrary sequence $\{f_j\}$, $j=1,2,\dotsc$, of complex numbers
satisfying the condition $\sup_j \abs{f_j}\le1$.
\end{theorem}

By taking $f_j=0$ for $j<l$ and $f_j=1$ for $j>l$ in~\eqref{f9}, we
obtain the following assertion.

\begin{corollary}\label{cor1}
Under the assumptions of the theorem, for arbitrary integer $l\ge1$
one has
\begin{equation*}
    \mathsf{P}_M\biggl(\Abs{\sum_{j=l}^{\infty}N_j-\sum_{j=l}^{\infty}\overline{N}_j}
               >\Delta\biggr)\le C_s\overline{N}^{-s},
    \qquad s=1,2,\dots.
\end{equation*}
\end{corollary}

Essentially, Theorem~\ref{t1} and Corollary~\ref{cor1} say that
for, large $M$, a random element $\{N_j\}$ in $\mathcal{X}_M$ is well
approximated by the Bose--Einstein distribution~\eqref{f6} with
parameter $b$. The following closed-form asymptotic expressions
relating $M$, $b$, $\overline{N}$, and the cumulative distribution
$\sum_{j\ge l}\overline{N}_j$ hold as $M\to\infty$ by
Proposition~\ref{euler1}, (i) in Sec.~\ref{sapp}:
\begin{equation}\label{p0}
    \begin{split}
M&=b^{-d-1}Q\Gamma(d+1)\zeta(d+1)(1+o(1)),\\
    \overline{N}&=b^{-d}Q\Gamma(d)\zeta(d)(1+o(1))\quad
\text{and hence}\\
    \overline{N}&=M^{\frac d{d+1}}Q^{\frac1{d+1}}\frac{\Gamma(d)\zeta(d)(1+o(1))}
    {(\Gamma(d+1)\zeta(d+1))^{\frac d{d+1}}},\\
    \sum_{j=l}^{\infty}\overline{N}_j&=Qb^{-d}\int_{bl}^\infty\frac{x^{d-1}\,dx}{e^x-1}
    +o(b^{-d}).
\end{split}
\end{equation}

The last formula in~\eqref{p0} reveals the role of the
nondimensionalized Bose--Einstein distribution function
\begin{equation}\label{fc3}
    \phi(x)=\frac{x^{d-1}}{e^x-1}.
\end{equation}
It is no surprise that, being appropriately normalized, the random
elements $\{N_j\}\in\Omega_M$, in a sense, tend as $M\to\infty$ to the
function~\eqref{fc3}. More precisely, let us define a sequence of
independent random functions $\phi_M(x)$, $M=1,2,\dotsc$, on the
positive real line $\mathbb{R}_+$ by setting
\begin{equation*}
    \phi_M(x)=Q^{-1}b^{d-1}N_j,\qquad x\in[b(j-1),bj),\quad
    j=1,2,\dotsc\,,
\end{equation*}
where $M$ and $b$ are related by \eqref{f5} and $\{N_j\}\in\Omega_M$
is a random element of the probability space $\mathcal{X}_M$. Then the
following assertion holds.
\begin{corollary}\label{cor2}
The sequence $\{\phi_M\}$ almost surely $*$-weakly converges in the
space $(C^1(\mathbb{R}_+))^*$ of continuous linear functionals on the
space $C^1(\mathbb{R}_+)$ to the function~\eqref{fc3}. Namely, for each
differentiable function $f(x)$, $x\in\mathbb{R}_+$, bounded together with
its first derivative uniformly on $\mathbb{R}_+$, one has
\begin{equation*}
    \langle\phi_M,f\rangle
    \xrightarrow{\text{a.s.}}\langle\phi,f\rangle,
\end{equation*}
where the angle brackets denote the pairing
\begin{equation*}
    \langle u,v\rangle=\int_0^\infty u(x)v(x)\,dx.
\end{equation*}
\end{corollary}
The convergence to the limit distribution~\eqref{fc3} can also be
stated in a somewhat different manner, as was done by
Vershik~\cite{Ver96}.
\begin{corollary}[cf.~{\cite[Theorem~4.4]{Ver96}}]\label{cor3}
For any $\varepsilon>0$ and any closed interval $[x_1,x_2]$,
$0<x_1<x_2<\infty$, there exists an $M_0$ such that
\begin{equation*}
    \mathsf{P}_M\biggl(\sup_{x\in[x_1,x_2]}
    \Abs{Q^{-1}b^d\sum_{j>x/b}N_j-\int_x^\infty\phi(x)\,dx}>\varepsilon\biggr)<\varepsilon
\end{equation*}
for $M>M_0$.
\end{corollary}
\begin{remark}
Note that the precise information on the convergence rate contained
in Theorem~\ref{t1} and Corollary~\ref{cor1} has been lost in
Corollaries~\ref{cor2} and~\ref{cor3}.
\end{remark}

\subsection{System with a fixed number of particles}\label{3.2}

Now let us consider systems in which two constraints, one on the
total energy and one on the number of particles, are given.

\subsubsection{Definition of the system}\label{3.2.1}
Let $M,N\ge0$ be integers. By~$\Omega_{M,N}$ we denote the set of all
sequences $\{N_j\}\equiv\{N_j\}_{j=0}^\infty$ of nonnegative
integers satisfying the conditions\footnote{It is merely a matter
of convenience that we have decided to start the indexing from
$j=0$, that is, have chosen zero for the ground energy level.
Should we wish to start from $j=1$, it suffices to change the
notation as follows: $\widetilde N=N$, $\widetilde M=M+N$, $\widetilde N_j=N_{j-1}$,
$j=1,2,\dotsc$. In terms of the variables with tildes, the sums
start from $j=1$.}
\begin{equation*}
    \sum_{j=0}^{\infty}N_j=N,\qquad \sum_{j=0}^{\infty} jN_j\le M.
\end{equation*}
Again, such sequences are finitely supported, and $\Omega_{M,N}$ is
finite. We take the same multiplicities $q_j>0$, $j=1,2,\dotsc$, as
in Sec.~\ref{3.1} and supplement them with some number $q_0\ge1$.
Next, we introduce the probability space
$\mathcal{X}_{M,N}=(\Omega_{M,N},\mathcal{F}_{M,N},\mathsf{P}_{M,N})$, where
$\mathcal{F}_{M,N}=2^{\Omega_{M,N}}$ and the probability $\mathsf{P}_{M,N}$ is defined
as follows:
\begin{equation}\label{v4}
    \mathsf{P}_{M,N}(\mathcal{A})=\frac{W(\mathcal{A})}{W(\Omega_{M,N})},
    \qquad \mathcal{A}\subset\Omega_{M,N},
\end{equation}
where, for every $\mathcal{A}\subset\Omega_{M,N}$, the weight
$W(\mathcal{A})$ is given by\footnote{We denote the weights in this
section by the capital letter $W$ so as to avoid confusion with the
weight $w$ introduced in Sec.~\ref{3.1} for sequences starting from
$j=1$.}
\begin{equation*}
    W(\mathcal{A})=\sum_{\{N_j\}\in\mathcal{A}}W(\{N_j\}),\quad
    W(\{N_j\})=\prod_{j=0}^{\infty}\binom{N_j+q_j-1}{N_j}.
\end{equation*}

\subsubsection{Limit distribution and Bose condensation}\label{3.2.2}

Let us study the behavior of random elements $\{N_j\}\in\Omega_{M,N}$
as $M\to\infty$ and $N\to\infty$. This problem involves two large
parameters $M$ and $N$ rather than one, and it is natural to expect
that the answer is more complicated than in the case of one large
parameter $M$, considered in Sec.~\ref{3.1}. It turns out that the
asymptotic behavior of our sequences $\{N_j\}$ strongly depend on
how the rates at which $M$ and $N$ tend to infinity are related.
Namely, let $\overline{N}=\overline{N}(M)$ be defined by~\eqref{f5} and~\eqref{f7}.
Recall that, by~\eqref{p0},
\begin{equation*}
    \overline{N}(M)=M^{\frac d{d+1}}Q^{\frac1{d+1}}\frac{\Gamma(d)\zeta(d)}
    {(\Gamma(d+1)\zeta(d+1))^{\frac d{d+1}}}(1+o(1)).
\end{equation*}
There are two possible types of asymptotic behavior of random
elements $\{N_j\}\in\Omega_{M,N}$ as $M,N\to\infty$ depending on
whether $N$ is smaller or greater than $\overline{N}$.
\begin{itemize}
    \item[(i)] If $N\le\overline{N}(M)$, then the limit distribution is the
    Bose--Einstein distribution
\begin{equation}\label{v7}
    \overline{N}_j=\frac{q_j}{e^{\beta j+\mu}-1},\qquad j=0,1,2,\dotsc\,,
\end{equation}
where the parameters $\beta,\mu>0$ are determined from the system of
equations
\begin{equation}\label{v7a}
    \sum_{j=0}^{\infty}\frac{q_j}{e^{\beta j+\mu}-1}=N,\qquad
    \sum_{j=0}^{\infty}\frac{jq_j}{e^{\beta j+\mu}-1}=M.
\end{equation}
    \item[(ii)] If $N>\overline{N}(M)$, then \textit{Bose condensation} occurs:
    the occupation numbers $\overline{N}_j$ with $j\ge1$ in the limit
    distribution no longer depend on $N$ and coincide with
    the numbers \eqref{f6}, while the excessive particles, however many,
    occupy the zero level,
\begin{equation}\label{v7b}
    \overline{N}_j=\frac{q_j}{e^{bj}-1},\quad j\ge1,\qquad
    \overline{N}_0=N-\sum_{j=1}^{\infty}\overline{N}_j=N-\overline{N}.
\end{equation}
\end{itemize}

Here we do not consider case~(i) in detail and refer the reader to
Theorem~5 in~\cite{MN2-2}, where the probabilities of deviations
from the limit distribution \eqref{v7}, \eqref{v7a} are estimated.

In case~(ii), the following theorem holds.

\begin{theorem}\label{t2}
Let the numbers $q_j$ satisfy \eqref{f1a}, and let $q_0\ge1$. Then
there exist constants $C_s>0$, $s=1,2,\dotsc$, such that the
estimates
\begin{equation}\label{v9}
    \mathsf{P}_{M,N}\biggl(\Abs{\sum_{j=0}^{\infty} f_j(N_j-\overline{N}_j)}
    >\Delta\biggr)\le C_s\overline{N}^{-s},
    \qquad s=1,2,\dots,
\end{equation}
where the $\overline{N}_j$ are defined in~\eqref{v7b} and $\Delta$
is the same as in Theorem~\textup{\ref{t1}}, hold for arbitrary
$M>0$ and $N>\overline{N}(M)$ and an arbitrary sequence $\{f_j\}$,
$j=0,1,2\dots$, of complex numbers satisfying the condition $\sup_j
\abs{f_j}\le1$.
\end{theorem}

In the same way as in Sec.~\ref{3.1.2}, we obtain the following
corollary.
\begin{corollary}\label{cor4}
Under the assumptions of the theorem, for arbitrary integer $l\ge0$
one has
\begin{equation*}
    \mathsf{P}_M\biggl(\Abs{\sum_{j=l}^{\infty}N_j-\sum_{j=l}^{\infty}\overline{N}_j}
               >\Delta\biggr)\le C_s\overline{N}^{-s},
    \qquad s=1,2,\dots.
\end{equation*}
\end{corollary}

In particular, this proves Theorem~\ref{t0}. Indeed, to derive the
estimate in that theorem from Corollary~\ref{cor4}, it suffices to
set $l=1$ and notice that
\begin{equation*}
    \abs{N_0-(N-\overline{N})}=\Abs{\sum_{j=1}^{\infty} N_j-\sum_{j=1}^{\infty}\overline{N}_j}.
\end{equation*}
\begin{remark}
One could also state the counterparts of Corollaries~\ref{cor2}
and~\ref{cor3}; we do not dwell upon this.
\end{remark}

\section{Proofs}\label{s4}

In this section, the letter $C$ is used to denote various
\textit{positive} constants independent of $N$, $b$, $M$, etc.
These constants are not assumed to be the same in all formulas! If
we need to keep track of several constants simultaneously, we equip
$C$ with subscripts. We also widely use the following standard
notation: we write $f\asymp g$ if the ratio $f/g$ is bounded above
and below by positive constants; in other words, $f$ and $g$ never
have opposite signs, $f=O(g)$, and $g=O(f)$.

\subsection{Proof of Theorem~\ref{t1}}\label{4.1}
Instead of~\eqref{f9}, it suffices to prove that (the modulus sign
is removed)
\begin{equation}\label{f9a}
    \mathsf{P}_M\biggl(\sum_{j=1}^{\infty} f_j(N_j-\overline{N}_j)>\Delta\biggr)\le C_s\overline{N}^{-s},
    \qquad s=1,2,\dots,
\end{equation}
where the $f_j$ are assumed to be real. Then we obtain \eqref{f9}
for real $f_j$ (with the constants $C_s$ multiplied by~$2$) by
combining \eqref{f9a} with the similar inequality where each $f_j$
has the same modulus and the opposite sign, and finally reach the
case of complex $f_j$ using Pythagoras' theorem (with further
increase in the constants). So we concentrate on the proof of
\eqref{f9a}.

By~\eqref{f4}, the probability on the left-hand side in \eqref{f9a}
is given by
\begin{equation*}
    \mathsf{P}_M\biggl(\sum_{j=1}^{\infty} f_j(N_j-\overline{N}_j)>\Delta\biggr)
    =\frac{w(\Omega_M(\Delta))}{w(\Omega_M)},
\end{equation*}
where $\Omega_M(\Delta)\subset\Omega_M$ is the set of all sequences
$\{N_j\}$ for which
\begin{equation}\label{f9e}
   \sum_{j=1}^{\infty} f_j(N_j-\overline{N}_j)>\Delta.
\end{equation}
To prove \eqref{f9a}, we will obtain a lower bound for $w(\Omega_M)$
and an upper bound for $w(\Omega_M(\Delta))$.

\subsubsection{A lower bound for $w(\Omega_M)$}
It is not so easy to estimate $w(\Omega_M)$ directly. Instead, we will
estimate the weight $w(\Omega_M^0)$, where $\Omega_M^0\subset\Omega_M$ is
the subset formed by the sequences for which $\sum_{j=1}^{\infty} jN_j=M$ (i.e.,
equality takes place in~\eqref{f2}). This weight obviously does not
exceed $w(\Omega_M)$. First, let us write out an exact formula for
$w(\Omega_M^0)$. Let
\begin{equation}\label{f13}
    F(z)=\sum_{M=0}^\infty w(\Omega_M^0)z^M
\end{equation}
be the generating function of the numbers $w(\Omega_M^0)$.
\begin{lemma}\label{le1}
The series \eqref{f13} converges in the disk $\{\abs{z}<1\}$, and
the sum is given by the formula
\begin{equation*}
    F(z)=\prod_{j=1}^{\infty}\frac{1}{(1-z^j)^{q_j}}.
\end{equation*}
\end{lemma}
This assertion is well known for integer $q_j$ (e.g., see
\cite{Ver96} and \cite[Chap.~1]{And76}), and the proof for
noninteger $q_j$ is also easy.\footnote{Here is the proof.
From~\eqref{f13}, using~\eqref{f3}, \eqref{f9d}, and the fact that
the sets $\Omega_M^0$ are disjoint, we obtain
\begin{equation*}
    F(z)
    =\sum_{M=0}^\infty
    \sum_{\{N_j\}\in\Omega_M^0}
         z^M\prod_{j=1}^{\infty}\binom{N_j+q_j-1}{N_j}
         =\sum_{\{N_j\}}\prod_{j=1}^{\infty}\biggl[\binom{N_j+q_j-1}{N_j} z^{jN_j}\biggr],
\end{equation*}
where the sum is taken over all finitely supported sequences
$\{N_j\}$ of nonnegative integers. Since the $q_j$ grow
polynomially by condition \eqref{f1}, it follows by routine
estimates that the sum and the product can be interchanged, and we
obtain
\begin{equation*}
    F(z)=\prod_{j=1}^{\infty}\biggl[\sum_{N_j=0}^\infty
         \binom{N_j+q_j-1}{N_j} z^{jN_j}\biggr]
         =\prod_{j=1}^{\infty}\frac{1}{(1-z^j)^{q_j}}
\end{equation*}
by the binomial series formula.}

Now we can express $w(\Omega_M^0)$ by the Cauchy formula
\begin{equation*}
    w(\Omega_M^0)=\frac1{2\pi i}\int_{\abs{z}=r}
             \frac{F(z)\,dz}{z^{M+1}}
            =\frac1{2\pi i}\int_{\abs{z}=r}
             \biggl[\prod_{j=1}^{\infty}\frac{1}{(1-z^j)^{q_j}}\biggr]\frac{dz}{z^{M+1}},
\end{equation*}
where $0<r<1$. The change of variables $z=e^{-\xi}$ yields
\begin{equation*}
    w(\Omega_M^0)=\frac1{2\pi i}\int_{\gamma_0}
               e^{\Phi(\xi)}\,d\xi,\qquad
    \Phi(\xi)=M\xi+\sum_{j=1}^{\infty} q_j\ln\frac1{1-e^{-j\xi}}.
\end{equation*}
Here the main branch of the logarithm is taken, and the integration
contour $\gamma_0$ is the circle $\{\operatorname{Re}\xi=-\ln r\}$ on the cylinder
$\mathbb{C}/2\pi i\mathbb{Z}$ with coordinate $\xi\mod2\pi i$. To estimate the
integral, we use a saddle-point argument
(cf.~\cite[Chap.~4]{Fed77}). Let us deform $\gamma_0$ into a contour
$\gamma^*$ on which
\begin{equation*}
    \min_\gamma\max_{\xi\in\gamma}\operatorname{Re}\Phi(\xi)
\end{equation*}
is attained, where the minimum is taken over all contours $\gamma$
lying in the right half-cylinder $\{\operatorname{Re}\xi>0\}$ and homotopic to
$\gamma_0$. We have
\begin{equation*}
    \operatorname{Re}\Phi(\xi)=M\operatorname{Re}\xi+\sum_{j=1}^{\infty} q_j\ln\frac1{\abs{1-e^{-j\xi}}}.
\end{equation*}
To find $\gamma^*$, the following lemma will be of help.
\begin{lemma}\label{le2}
\textup{(i)} For fixed $\operatorname{Re}\xi>0$, the maximum of $\operatorname{Re}\Phi(\xi)$ is
attained on the real axis.

\textup{(ii)} The minimum of $\operatorname{Re}\Phi(\xi)$ on the positive real axis
is attained at the point $\xi=b$, where $b$ is the root of
Eq.~\eqref{f5}.
\end{lemma}
\begin{proof}
(i) In fact, a stronger assertion is true:\footnote{Indeed, the
left-hand side of \eqref{f21} can be expressed as
\begin{equation*}
    \operatorname{Re}\Phi(\operatorname{Re}\xi)-\operatorname{Re}\Phi(\xi)=\sum_{j=1}^{\infty}
    q_j\ln\frac{\abs{1-e^{-j\xi}}}{1-e^{-j\operatorname{Re}\xi}}.
\end{equation*}
Let $x=j\operatorname{Re} \xi>0$ and $y=j\operatorname{Im}\xi$. Since $\ln v\ge 1-v^{-1}$ for
$v>1$, we have
\begin{equation*}
    \ln\frac{\abs{1-e^{-x-iy}}}{1-e^{-x}}\ge
    1-\frac{1-e^{-x}}{\abs{1-e^{-x-iy}}}=
    \frac{\sqrt{1+\delta}-1}{\sqrt{1+\delta}},
\quad
     \delta=\frac{2e^{-x}}{(1-e^{-x})^2}(1-\cos y).
\end{equation*}
Next,
\begin{equation*}
 \frac{\sqrt{1+\delta}-1}{\sqrt{1+\delta}}\ge
    \frac{\sqrt{1+\omega}-1}{\sqrt{1+\omega}},\quad
    \omega=2e^{-x}(1-\cos y)\le\delta.
\end{equation*}
It remains to note that $\omega\in[0,4]$, so that
\begin{equation*}
    \frac{\sqrt{1+\omega}-1}{\sqrt{1+\omega}}\ge
    \frac{\sqrt{1+\omega}-1}{\sqrt5}=
    \frac{\omega}{2\sqrt{1+\theta\omega}\sqrt{5}}
    \ge\frac\omega{10}=\frac15e^{-x}(1-\cos y)
\end{equation*}
(where $\theta\in[0,1]$), and we arrive at~\eqref{f21}.}
if $\operatorname{Re}\xi>0$, then
\begin{equation}\label{f21}
    \operatorname{Re}\Phi(\operatorname{Re}\xi)-\operatorname{Re}\Phi(\xi)\ge\frac15\sum_{j=1}^{\infty}
    q_je^{-j\operatorname{Re}\xi}\bigl(1-\cos(j\operatorname{Im}\xi)\bigr).
\end{equation}
All terms on the right-hand side in \eqref{f21} are nonnegative,
and the first term ($j=1$) is strictly positive unless
$\operatorname{Im}\xi\equiv0\mod2\pi$. This proves~(i).

(ii) For real $\xi$, we have $\operatorname{Re}\Phi(\xi)=\Phi(\xi)$. Next,
\begin{equation}\label{f22}
    \Phi'(\xi)=M-\sum_{j=1}^{\infty}\frac{jq_j}{e^{j\xi}-1},\qquad
    \Phi''(\xi)=\sum_{j=1}^{\infty}\frac{j^2q_je^{j\xi}}{(e^{j\xi}-1)^2}.
\end{equation}
The function $\Phi'(\xi)$ has the unique positive zero $\xi=b$
(cf.~\eqref{f5}), and this zero is the unique point of strict
minimum of $\operatorname{Re}\Phi(\xi)$ on the positive real axis, because the
second derivative is strictly positive for real $\xi$. This
gives~(ii) and completes the proof of Lemma~\ref{le2}.
\end{proof}
It follows from Lemma~\ref{le2} that for $\gamma^*$ we can take the
circle $\{\operatorname{Re}\xi=b\}$ on the cylinder $\mathbb{C}/2\pi
i\mathbb{Z}$.\footnote{Indeed, let $\gamma$ be any contour homotopic to
$\gamma_0$. Then $\gamma$ necessarily contains a point $\eta$ of the real
axis, and
\begin{equation*}
    \max_{\xi\in\gamma^*}\operatorname{Re}\Phi(\xi)
    \stackrel{(i)}{=}\Phi(b)
    \stackrel{(ii)}\le\Phi(\eta)\le\max_{\xi\in\gamma}\operatorname{Re}\Phi(\xi),
\end{equation*}
which shows that $\gamma^*$ is a saddle-point contour.} Hence we write
\begin{equation}\label{f23}
    w(\Omega_M^0)=\frac1{2\pi i}\int_{\gamma^*}e^{\Phi(\xi)}\,d\xi
    =\frac{\Phi(b)}{2\pi}\int_{-\pi}^\pi e^{S(\varphi)}\,d\varphi,
\end{equation}
where
\begin{equation*}
    S(\varphi)=\Phi(b+i\varphi)-\Phi(b)=
    iM\varphi +\sum_{j=1}^{\infty} q_j\ln\frac{1-e^{-bj}}{1-e^{-bj-ij\varphi}}.
\end{equation*}

Let us estimate the integral in~\eqref{f23}. By construction, we
have $S(0)=0$, and
\begin{equation}\label{f26}
    \operatorname{Re} S(\varphi)\le-\frac15\sum_{j=1}^{\infty}
    q_je^{-bj}\bigl(1-\cos(j\varphi)\bigr)
\end{equation}
by~\eqref{f21}. Next,
\begin{equation}\label{f25}
    S'(0)=i\Phi'(b)=0,
    \qquad S''(0)=-\Phi''(b)=
    -\sum_{j=1}^{\infty}\frac{j^2q_je^{bj}}{(e^{bj}-1)^2}
\end{equation}
(cf.~\eqref{f22}), and finally,
\begin{equation}\label{f27}
    S'''(\varphi)=i\sum_{j=1}^{\infty} \frac{j^3q_j(e^{2(bj+ij\varphi)}+e^{bj+ij\varphi})}
    {(e^{bj+ij\varphi}-1)^3}.
\end{equation}
\begin{lemma}\label{le3}
There exist positive constants $c_1,c_2,c_3$ such that
\begin{equation}\label{f28}
    -c_1b^{-d-2}\le S''(0)\le -c_2b^{-d-2},\qquad
    \sup_\varphi\abs{S'''(\varphi)}\le c_3b^{-d-3}
\end{equation}
\end{lemma}
\begin{proof}
We use inequalities \eqref{f1}. Then it follows from \eqref{f25}
and \eqref{f27} that
\begin{equation*}
      B_1\sum_{j=1}^{\infty}\frac{j^{d+1}e^{bj}}{(e^{bj}-1)^2}\le -S''(0)
         \le B_2\sum_{j=1}^{\infty}\frac{j^{d+1}e^{bj}}{(e^{bj}-1)^2}
\end{equation*}
and (since $\abs{e^{bj+ij\varphi}-1}\ge e^{bj}-1$)
\begin{equation*}
    \abs{S'''(\varphi)}\le B_2
    \sum_{j=1}^{\infty} \frac{j^{d+2}(e^{2bj}+e^{bj})}
    {(e^{bj}-1)^3}.
\end{equation*}
Now it remains to use Proposition \ref{euler1}, (ii)
in~Sec.~\ref{sapp}. This proves the lemma.
\end{proof}

We split the integration interval $[-\pi,\pi]$ in \eqref{f23} into
three zones,
\begin{gather*}
 D_1=\{\abs{\varphi}<\delta_1b^{1+d/3}\},\qquad
D_2=\{\delta_1b^{1+d/3}\le\abs{\varphi}\le \delta_2b\},\\
D_3=\{\delta_2b\le\abs{\varphi}\le\pi\}.
\end{gather*}
Here $\delta_1$ and $\delta_2$ are sufficiently small positive constants,
independent of $b$, to be chosen later. Let us represent $S(\varphi)$
by Taylor's formula with remainder of order~3,
\begin{equation*}
    S(\varphi)=\frac12S''(0)\varphi^2+R_3(\varphi),\quad\text{where}\quad
    \abs{R_3(\varphi)}\le\frac{c_3}6b^{-d-3}\abs{\varphi}^3
\end{equation*}
by the second inequality in \eqref{f28}. In particular,
\begin{equation}\label{f34}
    \abs{R_3(\varphi)}\le \frac{c_3\delta_1^3}6 \qquad\text{in $D_1$.}
\end{equation}
Take $\delta_3$ so small that the right-hand side in \eqref{f34} is
smaller than $\pi/4$. Then
\begin{equation*}
    c_4e^{-c_5b^{-d-2}\varphi^2}\le\operatorname{Re}
    e^{S(\varphi)}\le c_6e^{-c_7b^{-d-2}\varphi^2} \qquad\text{in $D_1$}
\end{equation*}
for some positive constants $c_4,\dots,c_7$. We have
\begin{equation*}
    \int_{D_1}e^{-cb^{-d-2}\varphi^2}\,d\varphi=
    b^{d/2+1}\int_{-\delta_1b^{-1/6}}^{\delta_1b^{-1/6}}e^{-cy^2}\,dy\asymp
    b^{d/2+1},
\end{equation*}
and hence
\begin{equation*}
    \int_{D_1}e^{S(\varphi)}\,d\varphi\asymp
    b^{d/2+1}.
\end{equation*}
Now consider the zone $D_2$. There we have
\begin{equation*}
\begin{split}
    \operatorname{Re} S(\varphi)&\le -(c_2b^{-d-2}-c_3b^{-d-3}\abs{\varphi})\varphi^2
    \\
    &\le
    -(c_2-\delta_2c_3)b^{-d-2}\varphi^2\le
    -(c_2-\delta_2c_3)b^{-d/3}\delta_1^2,
\end{split}
\end{equation*}
and if we take $\delta_2$ small enough that $c_8=c_2-\delta_2c_3>0$, then
the integral over $D_2$ decays exponentially (at the rate of
$e^{-c_8b^{-d/3}}$) as $b\to0$.

Finally, consider the zone $D_3$. Here we use inequality
\eqref{f26}. Since all terms in the series on the right-hand side
in \eqref{f26} are nonnegative, we can drop some terms and write
\begin{equation*}
    \operatorname{Re} S(\varphi)\le-\frac15\sum_{x_1<bj<x_2}
    q_je^{-bj}\bigl(1-\cos(j\varphi)\bigr)
\end{equation*}
with some positive $x_1$ and $x_2$ to be chosen later.
By~\eqref{f1}, $q_je^{-bj}\ge c_9b^{-d+1}$ for $x_1<bj<x_2$, where
the constant $c_9$ depends on $x_1$ and $x_2$, and so
\begin{equation*}
    \operatorname{Re} S(\varphi)\le-\frac{c_9b^{-d+1}}5
    \sum_{x_1<bj<x_2}\bigl(1-\cos(j\varphi)\bigr).
\end{equation*}
Using \cite[1.341.3]{GrRy63}, we obtain
\begin{equation}\label{f41}
\begin{split}
    \sum_{x_1<bj<x_2}&\bigl(1-\cos(j\varphi)\bigr)
    \\&=j_2-j_1+1+\frac{\cos\bigl((j_1+j_2)\varphi/2\bigr)
    \sin\bigl((j_2-j_1+1)\varphi/2\bigr)}
    {\sin(\varphi/2)}\\
    &\ge b^{-1}(x_2-x_1)-1-\abs{\sin(\varphi/2)}^{-1},
\end{split}
\end{equation}
where $j_1$ and $j_2$ are the first and the last integer,
respectively, in the interval $(x_1/b,x_2/b)$. Since $\varphi\in D_3$,
we have $\abs{\sin(\varphi/2)}\ge\sin(\delta_2b/2)\ge\delta_2b/\pi$, and so
the sum \eqref{f41} is greater than $c_{10}b^{-1}$ provided that we
take a sufficiently large $x_2-x_1$. Thus, we see that
\begin{equation*}
   \operatorname{Re} S(\varphi)\le c_{11}b^{-d}\qquad\text{in $D_3$},
\end{equation*}
so that the integral over $D_3$ also decays exponentially as
$b\to0$. Now we summarize the preceding and see that we have proved
the following assertion:
\begin{proposition}\label{p1}
There exists a constant $C>0$ such that
\begin{equation*}
    w(\Omega_M)\ge Cb^{d/2+1}e^{\mathcal{S}(M)},
\end{equation*}
where
\begin{equation}\label{f44}
    \mathcal{S}(M)=bM +\sum_{j=1}^{\infty} q_j\ln\frac1{1-e^{-bj}}
\end{equation}
and $b$ is related to $M$ by formula~\eqref{f5}.
\end{proposition}
This is the desired lower bound for $w(\Omega_M)$.

\subsubsection{An upper bound for $w(\Omega_M(\Delta))$}

It follows from \eqref{f2} and \eqref{f9e} that if
$\{N_j\}\in\Omega_M(\Delta)$ and $c\ge0$, then
\begin{equation}\label{f45}
    b\biggl(M-\sum_{j=1}^{\infty} jN_j\biggr)+c\sum_{j=1}^{\infty} f_j(N_j-\overline{N}_j)-c\Delta\ge0.
\end{equation}
Accordingly, the exponential of the left-hand side of \eqref{f45}
is greater than $1$ for any $\{N_j\}\in\Omega_M(\Delta)$; of course,
it is positive for any (finitely supported) $\{N_j\}$. Hence
\begin{equation}\label{f46}
\begin{split}
    w(\Omega_M(\Delta))&\le\sum_{\{N_j\}}
    \biggl[e^{b\bigl(M-\sum_{j=1}^{\infty} jN_j\bigr)+c\sum_{j=1}^{\infty} f_j(N_j-\overline{N}_j)-c\Delta}
    w(\{N_j\}\\
    &=e^{bM-c\Delta-c\sum_{j=1}^{\infty} f_j\overline{N}_j}
    \sum_{\{N_j\}}
    \prod_{j=1}^{\infty}\biggl[\binom{N_j+q_j-1}{N_j}
    e^{(-bj+cf_j)N_j}\biggr],
\end{split}
\end{equation}
where the sum extends over all finitely supported sequences
$\{N_j\}$ of nonnegative integers. Note that if $c\le b/2$, then
\begin{equation*}
    e^{-bj+cf_j}<\bigl(e^{-b/2}\bigr)^j
\end{equation*}
(recall that $\abs{f_j}<1$), and hence the series on the right-hand
side in~\eqref{f46} is dominated by the series~\eqref{f13} with
$z=e^{-b/2}<1$, which is convergent by Lemma~\ref{le1}. Arguing as
in the proof of Lemma~\ref{le1}, we interchange the sum and the
product and use the binomial formula to obtain
\begin{equation*}
\begin{split}
    w(\Omega_M(\Delta))&\le
    e^{bM-c\Delta-c\sum_{j=1}^{\infty} f_j\overline{N}_j}
    \prod_{j=1}^{\infty}\biggl[\sum_{N_j=0}^\infty\binom{N_j+q_j-1}{N_j}
    e^{(-bj+cf_j)N_j}\biggr]\\
        &=
        e^{bM-c\Delta-c\sum_{j=1}^{\infty} f_j\overline{N}_j}
        \prod_{j=1}^{\infty}\frac{1}{(1-e^{-bj+cf_j})^{q_j}}\\
        &=\exp\biggl\{bM-c\Delta
        +\sum_{j=1}^{\infty} q_j\biggl[\ln\frac{1}{1-e^{-bj+cf_j}}
        -\frac{cf_j}{e^{bj}-1}\biggr]\biggr\}.
\end{split}
\end{equation*}

By Taylor's formula with remainder,
\begin{equation*}
    \ln\frac{1}{1-e^{-bj+cf_j}}=
    \ln\frac{1}{1-e^{-bj}}
    +\frac{cf_j}{e^{bj}-1}
    +\frac{(cf_j)^2}2
    \frac{e^{bj-\theta_j cf_j}}{(e^{bj-\theta_j cf_j}-1)^2},
\end{equation*}
where $\theta_j\in[0,1]$. Since $\abs{f_j}\le1$ and $c\in[0,b/2]$,
we obtain, transposing the term $cf_j(e^{bj}-1)^{-1}$,
\begin{equation*}
\begin{split}
    \ln\frac{1}{1-e^{-bj+cf_j}}
        -\frac{cf_j}{e^{bj}-1}&\le \ln\frac{1}{1-e^{-bj}}
        +\frac{c^2}2
    \frac{e^{b(j+1/2)}}{(e^{b(j-1/2)}-1)^2}\\
    &\le \ln\frac{1}{1-e^{-bj}}
        +\frac{c^2}2
    \frac{e^be^{bj/2}}{(e^{bj/2}-1)^2},
\end{split}
\end{equation*}
and thus we arrive at the following assertion.
\begin{proposition}\label{p2}
One has the estimate
\begin{equation*}
    w(\Omega_M(\Delta))\le e^{\mathcal{S}(M)}e^{-c\Delta+c^2K},
    \qquad c\in[0,b/2],
\end{equation*}
where $\mathcal{S}(M)$ is given by~ \eqref{f44} and
\begin{equation*}
    K=\frac{e^b}2\sum_{j=1}^{\infty}\frac{q_je^{bj/2}}{(e^{bj/2}-1)^2}.
\end{equation*}
\end{proposition}

\subsubsection{Completion of the proof}

By combining Propositions~\ref{p1} and~\ref{p2}, we obtain
\begin{equation*}
    \mathsf{P}_M\biggl(\sum_{j=1}^{\infty} f_j(N_j-\overline{N}_j)>\Delta\biggr)\le
    Cb^{-1-d/2}e^{-c\Delta+c^2K},
    \qquad c\in[0,b/2].
\end{equation*}
Now we will prove that there exists a $c\in[0,b/2]$ such that
\begin{equation}\label{f52}
    e^{-c\Delta+c^2K}\le C_s\overline{N}^{-s},\qquad
    s=1,2,\dotsc\,.
\end{equation}
Then the assertion of the theorem readily follows, since $b\asymp
N^{-1/d}$.

We consider three cases.

\textbf{1.} $d>2$. In this case, $K<Cb^{-d}$ by
Proposition~\ref{euler1}, (ii) in Sec.~\ref{sapp}, and so $K\le
C\overline{N}$. Set
\begin{equation*}
    c=\frac1{2C}\overline{N}^{-1/2}\sqrt{\ln\overline{N}}\chi(\overline{N}).
\end{equation*}
Since $\overline{N}\asymp b^{-d}$, it follows that $c<b/2$ for sufficiently
large $M$. Next,
\begin{equation*}
    c\Delta-c^2K\ge \frac1{4C}\ln\overline{N}\chi(\overline{N}),
\end{equation*}
and \eqref{f52} holds.

\textbf{2.} $d=2$. In this case, $K\le Cb^{-2}\abs{\ln b}$
(Proposition~\ref{euler1}, (ii)) and hence $K\le C\overline{N}\ln\overline{N}$. Set
$c=b/2$. Then
\begin{equation*}
    c\Delta-c^2K\ge \frac b2\sqrt{\overline{N}}\ln\overline{N}\chi(\overline{N})
    -C\frac{b^2}{4}\overline{N}\ln\overline{N}.
\end{equation*}
The first term is of the order of $\ln\overline{N}\chi(\overline{N})$, and the
subtrahend is of the smaller order of $\ln\overline{N}$ , so that \eqref{f52}
again holds.

\textbf{3.} $1<d<2$. In this case, $K\le C b^{-2}$
(Proposition~\ref{euler1}, (ii)), and hence $K\le C\overline{N}^{2/d}$. We
again take $c=b/2$. Then
\begin{equation*}
    c\Delta-c^2K\ge \frac b2\overline{N}^{1/d}\ln\overline{N}\chi(\overline{N})
    -C\frac{b^2}{4}\overline{N}^{1/d}.
\end{equation*}
The first term is of the order of $\ln\overline{N}\chi(\overline{N})$, and the
subtrahend is $O(1)$, so that \eqref{f52} again holds.

The proof of Theorem \ref{t1} is complete. \qed

\subsection{Proof of Corollary~\ref{cor2}}

The proof of this corollary is based on the following lemma.
\begin{lemma}\label{le4}
One has
\begin{equation}\label{y0}
    \langle\phi-\phi_M,f\rangle
    =b^{-d}Q^{-1}\sum_{j=1}^{\infty}(N_j-\overline{N}_j)f(bj)+R_1(M)+R_2(M),
\end{equation}
where $R_1(M)$ is a random variable on $\Omega_M$ such that
\begin{equation}\label{y1}
    \abs{R_1(M)}\le Cb^{d+1}\biggl(\overline{N}-\sum_{j=1}^{\infty} N_j\biggr)
\end{equation}
and $R_2(M)$ is a \textup(nonrandom\textup) function such that
\begin{equation*}
    R_2(M)\to0\qquad\text{as}\quad M\to\infty.
\end{equation*}
\end{lemma}

\begin{proof}
Since the function $f(x)$ is bounded together with the first
derivative and $\phi(x)$ is given by \eqref{fc3} with $d>1$, it
follows from the Euler--Maclaurin formula (Proposition~\ref{euler})
that
\begin{equation*}
   \langle\phi,f\rangle\equiv \int_0^\infty \phi(x)f(x)\,dx=b\sum_{j=1}^{\infty} \phi(bj)f(bj)+r(M),
\end{equation*}
where $r(M)\to0$ as $M\to\infty$. Next, by \eqref{f1a},
\begin{equation*}
    b\phi(bj)\equiv b\frac{(bj)^{d-1}}{e^{bj}-1}
    =b^dQ^{-1}\overline{N}_j(1+\theta_j),
\end{equation*}
where $\theta_j\to0$ as $j\to\infty$. We claim that
\begin{equation*}
    b^dQ^{-1}\sum_{j=1}^{\infty}\overline{N}_j\theta_jf(bj)\to0\qquad\text{as $M\to\infty$}.
\end{equation*}
Indeed, since $f$ is uniformly bounded, we have
\begin{equation*}
    \Abs{b^dQ^{-1}\sum_{j=1}^{\infty}\overline{N}_j\theta_jf(bj)}\le
    Cb^d\biggl(\sum_{j=1}^k\overline{N}_j\abs{\theta_j}
    +\sum_{j=k+1}^\infty\overline{N}_j\abs{\theta_j}\biggr),
\end{equation*}
where $k$ is arbitrary. Take $k$ so large that
$\abs{\theta_j}<\varepsilon$ for $j>k$. Then the second sum does
not exceed $\varepsilon\overline{N}$. For \textit{fixed} $k$, the
first sum does not exceed $C_1b^{-1}$ (where the constant depends
on $k$) in view of formula \eqref{f6} for $\overline{N}_j$. Thus,
the right-hand side does not exceed
$CC_1b^{d-1}+Cb^d\overline{N}\varepsilon$. Recall that
$\overline{N}\asymp b^{-d}$; hence the second term can be made as
small as desired by an appropriate choice of $\varepsilon$, and
then we take $M$ large enough (i.e., $b$ small enough) to ensure
that the first term is also small. We conclude that
\begin{equation}\label{y1.5}
    \langle\phi,f\rangle=b^dQ^{-1}\sum_{j=1}^{\infty} \overline{N}_jf(bj)+r_1(M),
\end{equation}
where $r_1(M)\to0$ as $M\to\infty$.

Now let us study $\langle\phi_M,f\rangle$. We represent $f$ in the
form $f(x)=f_1(x)+f_2(x)$, where $f_1(x)$ is the step function
taking the value $f(bj)$ on the interval $[b(j-1),bj)$ and
$f_2(x)=f(x)-f_1(x)$ satisfies the estimate $\abs{f_2(x)}\le Cb$,
since the derivative $f'(x)$ is uniformly bounded. Then, in view of
the definition of $\phi_M(x)$, we have
\begin{equation*}
    \langle\phi_M,f\rangle
    =\langle\phi_M,f_1\rangle+\langle\phi_M,f_2\rangle
    =b^dQ^{-1}\sum_{j=1}^{\infty} N_jf(bj)+R(M),
\end{equation*}
where
\begin{align*}
    \abs{R(M)}&\le Cb\int_0^\infty\phi_M(x)\,dx=
    Cb^{d+1}Q^{-1}\sum_{j=1}^{\infty} N_j=Cb^{d+1}Q^{-1}\sum_{j=1}^{\infty}\overline{N}_j
    \\&\qquad{}+Cb^{d+1}Q^{-1}\biggl(N-\sum_{j=1}^{\infty}\overline{N}_j\biggr)\equiv r_2(M)+r_3(M).
\end{align*}
Now we can set $R_1(M)=r_3(M)$ and $R_2(M)=r_1(M)+r_2(M)$. The
proof of the lemma is complete.
\end{proof}

We assume without loss of generality that $\abs{f(x)}$ is bounded
by~$1$.

It follows from Theorem~\ref{t1} and Lemma~\ref{le4} that for each
$\varepsilon>0$ there exists an $M_0=M_0(\varepsilon)$ such that
\begin{equation}\label{y2}
    \mathsf{P}_M(\abs{\langle\phi-\phi_M,f\rangle}>\varepsilon)\le CM^{-2} \qquad
    \text{for $M\ge M_0$}.
\end{equation}
Indeed, the third term on the right-hand side in~\eqref{y0} is
necessarily less than $\varepsilon/3$ in modulus for sufficiently large~$M$.
Let us study the first term. It is necessarily less than $\varepsilon/3$ in
modulus provided that
\begin{equation*}
   \Abs{\sum_{j=1}^{\infty} f(bj)(N_j-\overline{N}_j)}\le C_0\overline{N},
\end{equation*}
where $C_0$ is some constant (depending on $Q$ and on the constants
in the relation $\overline{N}\asymp b^{-d}$. If $M$ is large enough that
$C_0\overline{N}>\Delta$ (which can always be achieved, because $\Delta$
grows slower than $\overline{N}$ as $M\to\infty$), we can apply
Theorem~\ref{t1} with appropriate $s$ and conclude that
\begin{equation*}
    \mathsf{P}_M\biggl(\Abs{b^{-d}Q^{-1}\sum_{j=1}^{\infty}(N_j-\overline{N}_j)f(bj)}>\varepsilon/3\biggr)\le
    CM^{-2}.
\end{equation*}
The same argument applies to the second term on the right-hand side
in~\eqref{y0}, and we arrive at \eqref{y2} (with some new $C$ and
$M_0(\varepsilon)$).

Consider the product $\prod_{M=1}^\infty\mathcal{X}_M$ of the probability
spaces $\mathcal{X}_M$. The probability measure on this product will be
denoted by $\mathsf{P}$. Consider the event
\begin{equation*}
    A_{m\varepsilon}=\bigl\{\abs{\langle\phi-\phi_M,f\rangle}\le\varepsilon
    \quad\text{for all $M\ge m$}\bigr\}.
\end{equation*}
It follows from \eqref{y2} that
\begin{equation*}
    \mathsf{P}(A_{m\varepsilon})\ge 1-C\sum_{M=m}^\infty M^{-2}
\end{equation*}
for $m\ge M_0(\varepsilon)$. The series on the right-hand side converges,
and so
\begin{equation*}
    \mathsf{P}(A_{m\varepsilon})\to1\qquad \text{as $m\to\infty$}.
\end{equation*}
Next, for the probability of the desired convergence
$\langle\phi_M,f\rangle\to\langle\phi,f\rangle$ we have the
expression
\begin{equation*}
    \mathsf{P}(\langle\phi-\phi_M,f\rangle\to0)=
    \mathsf{P}(\forall\varepsilon>0\exists m\colon A_{m\varepsilon})=
    \mathsf{P}\biggl(\bigcap_{\varepsilon>0}\bigcup_{m>0}A_{m\varepsilon}\biggr).
\end{equation*}
Note that the sets $A_{m\varepsilon}$ are nested,
\begin{equation*}
    A_{m\varepsilon}\subset A_{m'\varepsilon}\qquad\text{for $m\le m'$},
\end{equation*}
and that the sets
\begin{equation*}
    B_\varepsilon=\bigcup_{m>0}A_{m\varepsilon}
\end{equation*}
are nested as well,
\begin{equation*}
    B_\varepsilon\subset B_{\varepsilon'}\qquad\text{for $\varepsilon\le\varepsilon'$}.
\end{equation*}
It follows that
\begin{equation*}
    \mathsf{P}(B_\varepsilon)=\lim_{m\to\infty}\mathsf{P}(A_{m\varepsilon})=1,\qquad
    \mathsf{P}(\langle\phi-\phi_M,f\rangle\to0)=\lim_{\varepsilon\to0}\mathsf{P}(B_\varepsilon)=1.
\end{equation*}
The proof of Corollary \ref{cor2} is complete. \qed

\subsection{Proof of Corollary~\ref{cor3}}

Using Proposition~\ref{euler1} and arguing as in the proof of
\eqref{y1.5}, we obtain
\begin{equation*}
    \int_x^{\infty}\phi(x)\,dx=b^dQ^{-1}\sum_{j>x/b} \overline{N}_j +r(M),
\end{equation*}
where $r(M)\to0$ as $M\to\infty$. Hence it suffices to prove that
\begin{equation*}
    P_M\biggl(\sup_{x\in[x_1,x_2]}\Abs{\sum_{j>x/b}(N_j-\overline{N}_j)}>\varepsilon
    C\overline{N}\biggr)\le\varepsilon,
\end{equation*}
where $C$ is a constant depending on $Q$ and the constants in the
relation $\overline{N}\asymp b^{-d}$. For sufficiently large $M$, we have $\varepsilon
C\overline{N}>\Delta$, and Theorem~\ref{t1} can be used to estimate the
probability for each fixed $x$. This probability is less that
$C_s\overline{N}^{-s}$ for all~$s=0,1,2,\dotsc$. The supremum is actually
taken over the finite set of integer points on the interval
$[x_1/b,x_2/b]$, and the number of these points is of the order of
$1/b$, i.e., of the order of $\overline{N}^{1/d}$. Thus, passing from
individual points to the supremum over $O(\overline{N}^{1/d})$ points, we
make the estimate of the probability slightly worse (by the factor
equal to the number of these points); i.e., the probability
estimate becomes $\widetilde C_s\overline{N}^{-s+1/d}$ with some new constants $\widetilde
C_s$. However, this does not matter, and we take $s=1$, which
provides the desired estimate:
\begin{equation*}
    P_M\biggl(\sup_{x\in[x_1,x_2]}\Abs{\sum_{j>x/b}(N_j-\overline{N}_j}>\varepsilon
    C\overline{N}\biggr)\le \widetilde C_1\overline{N}^{-1+1/d}\le \varepsilon
\end{equation*}
for sufficiently large $M$ (and hence $N$).

The proof of Corollary~\ref{cor3} is complete.

\subsection{Proof of Theorem~\ref{t2}}

Just as in the proof of Theorem~\ref{t1}, it suffices to show that
\begin{equation}\label{z1}
    \mathsf{P}_{M,N}\biggl(\sum_{j=0}^{\infty} f_j(N_j-\overline{N}_j)
    >\Delta\biggr)\le C_s\overline{N}^{-s},
    \qquad s=1,2,\dots,
\end{equation}
i.e, drop the modulus sign. Next, we will assume without loss of
generality that $f_0=0$. Indeed,
\begin{equation*}
    \sum_{j=0}^{\infty}N_j=\sum_{j=0}^{\infty}\overline{N}_j=N,
\end{equation*}
and hence
\begin{equation*}
    \sum_{j=0}^{\infty} f_j(N_j-\overline{N}_j)=\sum_{j=1}^{\infty} (f_j-f_1)(N_j-\overline{N}_j).
\end{equation*}
The new numbers $f_j'=f_j-f_1$ are bounded in absolute value by~$2$
rather than~$1$, but we can pass to the numbers $\widetilde f_j=f_j'/2$
and simultaneously divide the function $\chi(x)$ occurring in
definition of $\Delta$ by~$2$.

By definition \eqref{v4}, the probability on the left-hand side in
\eqref{z1} has the form
\begin{equation}\label{z2}
    \mathsf{P}_{M,N}\biggl(\sum_{j=1}^{\infty} f_j(N_j-\overline{N}_j)>\Delta\biggr)
    =\frac{W(\Omega_{M,N}(\Delta))}{W(\Omega_{M,N})},
\end{equation}
where $\Omega_{M,N}(\Delta)\subset\Omega_{M,N}$ is the set of all
sequences $\{N_j\}_{j=0}^\infty\in\Omega_{M,N}$ such that
\begin{equation}\label{z4}
    \sum_{j=1}^{\infty} f_j(N_j-\overline{N}_j)>\Delta.
\end{equation}

Consider the mapping
\begin{equation*}
    \tau\colon\Omega_{M,N}\longrightarrow\Omega_M
\end{equation*}
that takes each sequence
$\{N_j\}=\{N_j\}_{j=0}^\infty\in\Omega_{M,N}$ to the sequence
$\{N_j\}'=\{N_j\}_{j=1}^\infty\in\Omega_M$ obtained by throwing
away the first element~$N_0$. This mapping is one-to-one onto its
range, because $N_0=N-\sum_{j=1}^{\infty} N_j$ is uniquely
determined by $\{N_j\}'$. Since $f_0=0$, it follows that
\begin{equation*}
    \tau(\Omega_{M,N}(\Delta))\subset \Omega_M(\Delta).
\end{equation*}
(Inequality \eqref{z4} does not involve $N_0$.)

For each $\{N_j\}\in\Omega_{M,N}$, we have
\begin{equation*}
    W(\{N_j\})=\binom{N_0+q_0-1}{N_0} w(\{N_j\}').
\end{equation*}
Since, by assumption, $q_0>1$, we see that the binomial $\binom{N_0+q_0-1}{N_0}$
is a monotone increasing function of~$N_0$; by Stirling's formula,
\begin{equation*}
    \binom{N_0+q_0-1}{N_0}\asymp N_0^{q_0-1}.
\end{equation*}

Now we can write
\begin{align*}
    W(\Omega_{M,N}(\Delta))
    &=\sum_{\{N_j\}\in\Omega_{M,N}(\Delta)}
    \binom{N_0+q_0-1}{N_0} w(\{N_j\}')
    \\&\le
    \biggl\{\max_{\Omega_{M,N}(\Delta)}\binom{N_0+q_0-1}{N_0}\biggr\}
    \sum_{\{N_j\}\in\Omega_{M,N}(\Delta)}
    w(\{N_j\}')
    \\&\le
    CN^{q_0-1}w(\Omega_M(\Delta)).
\end{align*}
On the other hand,
\begin{align*}
    W(\Omega_{M,N})
    &=\sum_{\{N_j\}\in\Omega_{M,N}}
    \binom{N_0+q_0-1}{N_0} w(\{N_j\}')
    \\&\ge
    \biggl\{\min_{\Omega_{M,N}}\binom{N_0+q_0-1}{N_0}\biggr\}
    \sum_{\{N_j\}\in\Omega_{M,N}}
    w(\{N_j\}')
    \\&=
    \biggl\{\min_{\Omega_{M,N}}\binom{N_0+q_0-1}{N_0}\biggr\}w(\tau(\Omega_{M,N})).
\end{align*}

We have $N>\overline{N}$. Let us consider two cases, $N>\overline{N}+\Delta$ and $\overline{N}\le
N<\overline{N}+\Delta$.

\textbf{1.} Let $N>\overline{N}+\Delta$. The set
$\Omega_M\setminus\tau(\Omega_{M,N})$ consists of all sequences
$\{N_j\}\in\Omega_M$ for which
\begin{equation*}
    \sum_{j=1}^{\infty}(N_j-\overline{N}_j)=\sum_{j=1}^{\infty} N_j-\overline{N}>\Delta.
\end{equation*}
Hence
\begin{equation*}
    \frac{w(\Omega_M)-w(\tau(\Omega_{M,N}))}{w(\Omega_M)}\equiv
    \mathsf{P}_M(w(\Omega_M)-w(\tau(\Omega_{M,N})))\le C_s\overline{N}^{-s},
\end{equation*}
$s=1,2,\dotsc$, by Corollary \ref{cor1} (with $l=1$). It follows
that, for sufficiently large $M$ and $N\ge\overline{N}(M)+\Delta$,
\begin{equation*}
    W(\Omega_{M,N})\ge
    \frac12 \biggl\{\min_{\Omega_{M,N}}\binom{N_0+q_0-1}{N_0}\biggr\}w(\Omega_M)
\end{equation*}
and
\begin{multline}\label{x-files}
    \mathsf{P}_{M,N}\biggl(\sum_{j=1}^{\infty} f_j(N_j-\overline{N}_j)>\Delta\biggr)
    \\
    \le \frac{2CN_0^{q_0-1}}{\min_{\Omega_{M,N}}\binom{N_0+q_0-1}{N_0}}
    \mathsf{P}_M\biggl(\sum_{j=1}^{\infty} f_j(N_j-\overline{N}_j)>\Delta\biggr)
    \\
    \le C_s\frac{N_0^{q_0-1}}{\min_{\Omega_{M,N}}\binom{N_0+q_0-1}{N_0}}\overline{N}^{-s}.
\end{multline}
If $N<M$, then the minimum is attained at $N_0=0$ and is equal
to~$1$. If $N\ge M$, then the minimum is attained at  $N_0=N-M$,
and we have, by Stirling's formula,
\begin{equation*}
    \frac{N^{q_0-1}}{\min_{\Omega_{M,N}}\binom{N_0+q_0-1}{N_0}}\asymp
    \biggl(\frac{N}{N-M}\biggr)^{q_0-1}\le CM^{q_0-1}\le\widetilde
    CN^{(q_0-1)(d+1)/d}.
\end{equation*}
In any case, we obtain
\begin{equation*}
    \mathsf{P}_{M,N}\biggl(\sum_{j=1}^{\infty} f_j(N_j-\overline{N}_j)>\Delta\biggr)\le C_s\overline{N}^{-s}
\end{equation*}
(with some new constants $C_s$). This proves the theorem for this
case.

\textbf{2.} Let $\overline{N}<N\le\overline{N}+\Delta$. Here we need a finer argument.

We carry out the same computations for the upper bound of the
numerator, but to estimate the denominator, we write
\begin{equation*}
\begin{split}
    W(\Omega_{M,N})&\ge\binom{N-\overline{N}+q_0-1}{N-\overline{N}}
    \sum_{\{N_j\}'\in\Omega'_{M,\overline{N}}}w(\{N_j\})
    \\&=\binom{N-\overline{N}+q_0-1}{N-\overline{N}}w(\Omega'_{M,\overline{N}}),
\end{split}
\end{equation*}
where $\Omega'_{M,\overline{N}}\subset\Omega_M$ is the subset of sequences
satisfying the conditions
\begin{equation*}
    \sum_{j=1}^{\infty} N_j=\overline{N},\qquad\sum_{j=1}^{\infty} jN_j=M.
\end{equation*}
It follows from Lemma~6 and formula~(40) in~\cite{MN2-2} that, for
some $m_0$,
\begin{equation*}
    w(\Omega'_{M,\overline{N}})\ge Ce^{\mathcal{S}(M)}\overline{N}^{-m_0},
\end{equation*}
where $\mathcal{S}(M)$ is defined in \eqref{f44}. By combining this with
the estimate for $w(\Omega_M(\Delta))$ obtained in the proof of
Theorem~\ref{t1} (Proposition~\ref{p2}), we arrive at the desired
estimate.

The proof of Theorem~\ref{t2} is complete.

\section{Appendix}\label{sapp}

We shall use the following two versions of the Euler--Maclaurin
formula.
\begin{proposition}\label{euler}
Let $f(x)$ be a continuously differentiable function on
$(0,\infty)$ absolutely integrable together with $f'(x)$ at
infinity. Then the series $\sum_{j=1}^{\infty} f(j)$ converges absolutely, and its
sum can be computed by the formula
\begin{equation}\label{eu1}
    \sum_{j=1}^{\infty} f(j)=\int_1^\infty f(x)\,dx +R_1,
\end{equation}
where the remainder $R_1$ satisfies the estimate
\begin{equation*}
    \abs{R_1}\le\int_1^\infty\abs{f'(x)}\,dx.
\end{equation*}
If, in addition, $f'(x)$ is absolutely integrable at zero, then
\begin{equation}\label{eu2}
    \sum_{j=1}^{\infty} f(j)=\int_0^\infty f(x)\,dx +R_2,
\end{equation}
where the remainder $R_2$ satisfies the estimate
\begin{equation*}
    \abs{R_2}\le\int_0^\infty\abs{f'(x)}\,dx.
\end{equation*}
\end{proposition}
\begin{proof}
The obtain the desired error estimates, one can use the Stieltjes
integral representations
\begin{equation*}
    \sum_{j=1}^{\infty} f(j)=\int_0^\infty f(x)\,d[x]=-\int_1^\infty f(x)\,d[-x],
\end{equation*}
where $[x]$ is the integer part of $x$, and then integrate by parts
to estimate the remainders.
\end{proof}

Let us use Proposition~\ref{euler} to carry out some computations
needed in the main text.

\begin{proposition}\label{euler1}
The following asymptotic formulas hold as $b\to+0$.

\textup{(i)} Let $s>0$. Then
\begin{equation*}
    \sum_{j=1}^{\infty} \frac{j^s}{e^{bj}-1}=b^{-s-1}\Gamma(s+1)\zeta(s+1)(1+o(1))
\end{equation*}
and, more generally,
\begin{equation*}
    \sum_{j=l}^{\infty} \frac{j^s}{e^{bj}-1}=
    b^{-s-1}\int_{bl}^\infty
    \frac{y^s\,dy}{e^y-1}+o(b^{-s-1}).
\end{equation*}

\textup{(ii)} Let $d>1$. Then
\begin{align*}
    \sum_{j=1}^{\infty}\frac{j^{d+1}e^{bj}}{(e^{bj}-1)^2}&=
    b^{-d-2}\int_{0}^\infty
    \frac{y^{d+1}e^y\,dy}{(e^y-1)^2}(1+o(1)),\\
     \sum_{j=1}^{\infty}\frac{j^{d+2}(e^{2bj}+e^{bj})}{(e^{bj}-1)^3}&=
    b^{-d-3}\int_{0}^\infty
    \frac{y^{d+2}(e^{2y}+e^{y})\,dy}{(e^y-1)^3}(1+o(1)),\\
    \sum_{j=1}^{\infty}\frac{j^{d-1}e^{bj/2}}{(e^{bj/2}-1)^2}&\le
    \begin{cases}
    Cb^{-d} &\text{if $d>2$,}\\
    Cb^{-2}\abs{\ln b} &\text{if $d=2$,}\\
    Cb^{-2}&\text{if $1<d<2$.}
    \end{cases}
\end{align*}
\end{proposition}
\begin{proof}
(i) If $s>1$, then, by \eqref{eu2},
\begin{align*}
    \sum_{j=1}^{\infty} \frac{j^s}{e^{bj}-1}&=\int_0^\infty
    \frac{x^s\,dx}{e^{bx}-1}+R_2=
    b^{-s-1}\int_0^\infty
    \frac{y^s\,dy}{e^y-1}+R_2\\&=
    b^{-s-1}\Gamma(s+1)\zeta(s+1)+R_2
\end{align*}
(we have used \cite[3.411.1]{GrRy63}), where
\begin{equation*}
    \abs{R_2}\le
    b^{-s}\int_0^\infty\frac{sy^{s-1}(e^{y}-1)
    +y^se^y}{(e^{y}-1)^2}\,dy,
\end{equation*}
and we have the desired estimate, because the integral converges.
If $0<s\le1$, then this computation does not work, but we can use
\eqref{eu1} and write
\begin{align*}
    \sum_{j=1}^{\infty} \frac{j^s}{e^{bj}-1}&=\int_1^\infty
    \frac{x^s\,dx}{e^{bx}-1}+R_1=
    b^{-s-1}\int_b^\infty
    \frac{y^s\,dy}{e^y-1}+R_1\\
    &=
    b^{-s-1}\int_0^\infty
    \frac{y^s\,dy}{e^y-1}-
    b^{-s-1}\int_0^b
    \frac{y^s\,dy}{e^y-1}+R_1,
\end{align*}
where
\begin{equation*}
    \abs{R_1}\le
    b^{-s}\int_b^\infty\frac{sy^{s-1}(e^{y}-1)
    +y^se^y}{(e^{y}-1)^2}\,dy,
\end{equation*}
and so $\abs{R_1}\le Cb^{-1}$ for $s\in(0,1)$ and $\abs{R_1}\le
Cb^{-1}\abs{\ln b}$ for $s=1$. Now note that
\begin{equation*}
    \int_0^b
    \frac{y^s\,dy}{e^y-1}\le Cb^s,
\end{equation*}
and we obtain the desired estimate.

The estimate for the sum starting form $j=l$ can be obtained in a
similar way; one only need to shift the summation index by $l-1$.

(ii) The proof goes by the same scheme as for (i). We omit the
details.
\end{proof}
\begin{remark}
The integrals in (ii) can also be expressed via special functions,
but we do not need these expressions.
\end{remark}

\end{document}